\documentclass[12pt]{article}
\usepackage{setspace}
\usepackage{graphicx} 
\usepackage{amsthm}
\usepackage{amsmath}
\usepackage{mdframed}
\usepackage{amssymb}
\usepackage{enumitem}
\usepackage{pict2e}
\usepackage[margin=1in]{geometry}
\usepackage{mathtools}
\usepackage{xcolor}
\usepackage[colorlinks=true]{hyperref}
\hypersetup{
    colorlinks=true,
    linkcolor=blue,
    filecolor=magenta,      
    urlcolor=cyan,
    linktoc=all, 
}
\usepackage{cancel}
\makeatletter
\newcommand{\bigcomp}{%
  \DOTSB
  \mathop{\vphantom{\sum}\mathpalette\bigcomp@\relax}%
  \slimits@
}
\newcommand{\bigcomp@}[2]{%
  \begingroup\m@th
  \sbox\z@{\(#1\sum\)}%
  \setlength{\unitlength}{0.9\dimexpr\ht\z@+\dp\z@}%
  \vcenter{\hbox{%
    \begin{picture}(1,1)
    \bigcomp@linethickness{#1}
    \put(0.5,0.5){\circle{1}}
    \end{picture}%
  }}%
  \endgroup
}
\newcommand{\bigcomp@linethickness}[1]{%
  \linethickness{%
      \ifx#1\displaystyle 2\fontdimen8\textfont\else
      \ifx#1\textstyle 1.65\fontdimen8\textfont\else
      \ifx#1\scriptstyle 1.65\fontdimen8\scriptfont\else
      1.65\fontdimen8\scriptscriptfont\fi\fi\fi 3
  }%
}

\makeatother
\usepackage{tikz}
\usepackage{natbib}
\setcitestyle{authoryear,open={(},close={)}}
\usepackage{stmaryrd}
\usepackage{colonequals}

\newtheorem{theorem}{Theorem}[section]
\makeatletter
\renewenvironment{proof}[1][\proofname]{\par
  \pushQED{\qed}%
  \normalfont \topsep6\p@\@plus6\p@\relax
  \list{}{\leftmargin=3em
          \rightmargin=\leftmargin
          \settowidth{\itemindent}{\itshape#1}%
          \labelwidth=\itemindent
          \parsep=0pt \listparindent=\parindent 
  }
  \item[\hskip\labelsep
        \itshape
    #1\@addpunct{.}]\ignorespaces
}{%
  \popQED\endlist\@endpefalse
}
\newtheorem{lemma}[theorem]{Lemma}
\theoremstyle{definition}
\newtheorem{prop}{Proposition}[section]
\newtheorem{cor}{cor}

\theoremstyle{definition}
\newenvironment{subproof}[1][\proofname]{%
  \begin{proof}[#1]%
}{%
  \end{proof}%
}
\newenvironment{aside}
  {\begin{mdframed}[style=0,%
      leftline=false,rightline=false,leftmargin=2em,rightmargin=2em,%
          innerleftmargin=0pt,innerrightmargin=0pt,linewidth=0.75pt,%
      skipabove=7pt,skipbelow=7pt]\small}
  {\end{mdframed}}
\newtheorem{definition}{definition}[section]
\title{Foundations of Quantum Contextual Topos: Integrating Modality and Topos Theory in Quantum Logic}
\author{Jesse Werbow}
\date{August 2024}
\usepackage{times}
\begin{document}

\maketitle

\section*{Abstract}
This paper introduces the Quantum Contextual Topos (QCT), a novel framework that extends traditional quantum logic by embedding contextual elements within a topos-theoretic structure. This framework seeks to provide a classically-obedient tool for exploring the logical foundations of quantum mechanics. The QCT framework aims to address the limitations of classical quantum logic, particularly its challenges in capturing the dynamic and contextual nature of quantum phenomena. By integrating modal operators and classical propositional logic within a topos structure, the QCT offers a unified approach to modeling quantum systems. The main result of this work is demonstrating that the internal logic of QCT corresponds to a form of classical propositional polymodal logic. We do this by generalizing Stone's Representation Theorem for a specific case of polymodal algebras and their underlying Stone Spaces.

\tableofcontents

\section{Introduction}

The intersection of mathematical logic and quantum mechanics, particularly in the context of quantum gravity, has garnered increasing attention in recent years \citep{causets, majid2006algebraic, massimo2022, griffiths1994quantum, rickles2016quantum, catianeo1993}. This renewed interest stems from the realization that traditional physical models, which are deeply rooted in classical concepts such as determinism and continuity, may not sufficiently capture the complexities of quantum phenomena. Quantum mechanics introduces several challenges, such as probabilistic outcomes and non-commutative observables. These challenges suggest that alternative logical frameworks may be crucial for advancing our understanding of quantum theory. Consequently, there is a growing recognition that such frameworks could play a vital role in the development of quantum gravity.

In classical physics, logical structures such as Boolean algebra provide a robust foundation for reasoning about physical systems. However, quantum mechanics defies these classical logical structures in several fundamental ways. For instance, the superposition principle and the inherent uncertainty in quantum measurements cannot be easily reconciled with classical logic's distributive properties. The necessity of a new logical framework is particularly pronounced in the field of quantum gravity, where attempts to unify general relativity and quantum mechanics have repeatedly encountered conceptual and technical obstacles.

The search for alternative logical structures is not merely theoretical but has practical implications in quantum gravity. A notable example is \textit{causal set theory} (causets), which suggests that spacetime is fundamentally discrete rather than continuous. This approach aligns with the broader need to develop quantum gravity theories that diverge from traditional, continuous models. In this theory, the fabric of spacetime is composed of elementary "causal sets," with the relationships between these sets governed by a logical structure that is distinct from the continuous geometric spaces of classical general relativity \citep{causets}. This shift from a continuum to a discrete model not only challenges traditional ontological assumptions but also necessitates a reevaluation of the logical frameworks that underpin our physical theories.

The implications of adopting a logic-centered perspective in quantum gravity are profound, extending beyond mere technical modifications. Such approaches challenge the foundational assumptions of classical physics, prompting a reconsideration of the very nature of reality as described by quantum theory. The movement towards more abstract and formal methodologies reflects a broader trend within theoretical physics, where the integration of quantum mechanics with mathematical logic is increasingly seen as a critical avenue for resolving long-standing issues.

This paper contributes to this ongoing discourse by exploring the potential of a novel logical framework, which we term the "Quantum Contextual Topos" (QCT). The QCT framework seeks to address the limitations of traditional quantum logic by embedding quantum logical propositions within a topos-theoretic structure, thereby offering a more comprehensive treatment of quantum phenomena. By integrating elements of modal logic within this topos structure, the QCT aims to bridge the conceptual gap between quantum mechanics and classical logic, providing a unified approach to modeling quantum systems.

\subsection{Motivation and Background}

Traditional quantum logic, as formalized by Von Neumann and others, represents quantum measurement outcomes using projection operators within a propositional framework \citep{nine}. These operators project quantum states onto eigenspaces corresponding to specific measurement outcomes. However, the algebraic structure of these propositions, typically organized as an orthomodular lattice, deviates significantly from classical logic, particularly in its handling of logical operations such as conjunction and disjunction \citep{great}. The non-distributive nature of this lattice structure presents substantial challenges in describing quantum phenomena within the confines of classical logic.

One of the central issues with traditional quantum logic is its focus on static properties, which limits its ability to capture the dynamic aspects of quantum mechanics, such as basis transformations \citep{prat}. Quantum logic departs from the classical logic’s distributive law, which states:
\[
A\land(B\lor C)=(A\land B)\lor(A\land C).
\]
where \(\land\) denotes conjunction and \(\lor\) denotes disjunction. This departure is especially evident in the non-commutativity of quantum measurements, which complicates classical reasoning about logical operations. The uncertainty principle, for instance, demonstrates how the conjunction of propositions related to momentum and position cannot be distributed in a manner consistent with classical logic. The precision limitations inherent in quantum measurements result in logical operations that defy classical intuition.

Moreover, traditional Neumannian quantum logic struggles with scenarios where different measurement bases are conflated, leading to counter-intuitive results. Consider, for example, a particle moving along a one-dimensional line, where the reduced Planck constant \(\hbar=1\). Let us define the following propositions:
\begin{itemize}
    \item \(P\): "The particle has momentum in the interval \([-\frac{1}{8}, +\frac{1}{8}]\)."
    \item \(Q\): "The particle has position in the interval \([-2,0]\)."
    \item \(R\): "The particle has position in the interval \([0,+2]\)."
\end{itemize}
According to the uncertainty principle (\(\Delta x\Delta p\geq\frac{1}{2}\)), the expression \(P\land(Q\lor R)\) holds true, suggesting that the particle's momentum and position can simultaneously fall within the specified intervals without violating quantum mechanical constraints. However, the distributed version \((P\land Q)\lor(P\land R)\) is false, as it implies a level of precision in both momentum and position that contradicts the uncertainty principle. Specifically, the uncertainty product in this scenario is \(\frac{1}{4}\), which is less than the minimum value of \(\frac{1}{2}\) required by quantum mechanics. This example underscores the limitations of traditional quantum logic in reconciling different measurement bases and highlights the need for a more nuanced logical framework.

\subsection{Modal Logic}

Modal logic is an extension of classical propositional logic that introduces the concept of modality, allowing statements to express more than just truth or falsity. In modal logic, propositions can be qualified by operators that capture notions such as necessity (denoted by \(\square\)) and possibility (denoted by \(\Diamond\)). For example, if \(p\) is a proposition, \(\square p\) means "it is necessarily true that \(p\)," while \(\Diamond p\) means "it is possibly true that \(p\)" \citep{blackburn_rijke_venema:modal_logic}.

The semantics of modal logic are often framed using Kripke models, which consist of a set of possible worlds, a relation between these worlds (called an accessibility relation), and an assignment of truth values to each proposition in each world. Formally, a Kripke frame is a pair \((W,R)\), where \(W\) is a non-empty set of possible worlds, and \(R\subseteq W\times W\) is the accessibility relation that connects these worlds. The accessibility relation defines which worlds are "accessible" from any given world, thereby determining the truth of modal statements based on their truth in accessible worlds. For instance, \(\square p\) is true in a world \(w\) if \(p\) is true in all worlds accessible from \(w\); similarly, \(\Diamond p\) is true in \(w\) if \(p\) is true in at least one world accessible from \(w\).

In classical settings, these structures allow for reasoning about how truth values change or hold across different contexts or scenarios. However, when applied to quantum mechanics, modal logic may offer a framework that aligns with the inherent contextuality and non-commutative nature of quantum observables. The use of Kripke frames and accessibility relations becomes particularly valuable in modeling the relationships between different quantum states and the logical operations that apply within those states.

\subsection{Topos-Theoretic Approach to Quantum Physics}

To address the limitations of traditional quantum logic, we propose the introduction of a topos structure, which we term the "Quantum Contextual Topos" (QCT). The QCT framework extends the concept of a frame in modal logic to accommodate the peculiarities of quantum logic, offering a more comprehensive treatment of quantum phenomena. By integrating multiple modal operators within this framework, the QCT enriches the logical structure, allowing for a more detailed exploration of quantum mechanics.

In categorical theory, a topos serves as a versatile framework, akin to a "universe" where objects and morphisms interact under a specific set of logical rules \citep{johnstone2002topos}. These rules, known as the \textit{internal logic} of the topos, can be tailored to reflect the unique principles of the mathematical structures they organize. Within the QCT framework, we define a "Quantum Frame" as a type of general frame in modal logic, serving as the foundational structure of the QCT. Our analysis reveals that the internal logic of the QCT aligns with propositional polymodal logic, making it a powerful tool for exploring the logical foundations of quantum mechanics.

The concept of general frames extends modal semantics beyond traditional Kripke frames \citep{bellis1993}. As a form of topos, general frames provide a sophisticated mathematical structure that is well-suited for delving into the logical underpinnings of quantum mechanics \citep{maclane1994sheaves}. By constructing a topos for quantum mechanics, we tailor the internal logic to reflect the principles of quantum theory, capturing the subtleties of phenomena such as superposition and entanglement \citep{isham1997topos, heunen2009bohrification, Landsman2017}. Quantum frames, as a variant of polymodal general frames, are the cornerstone of the QCT, enabling the formalization of contextual propositions on an orthomodular lattice with indexed modal operators.

\subsubsection*{Comparison to Bohrification}

The internal logic of the QCT is intentionally classical in design, closely resembling the polymodal logic that we seek to articulate. This classical framework is crucial for maintaining consistency with quantum mechanical principles while providing a structured approach to investigating quantum propositions. When compared to Bohrification, another topos-theoretic approach to quantum mechanics, QCT offers both similarities and distinctions. Bohrification emphasizes the contextual nature of quantum observables, drawing on Niels Bohr's philosophical views on quantum theory. Both QCT and Bohrification use topos theory to provide a mathematical framework for quantum mechanics, but they differ in their treatment of internal logic. While Bohrification typically employs an intuitionistic internal logic to reflect the non-classical nature of quantum logic, QCT endeavors to retain a classical internal logic within the topos, offering an alternative interpretation of quantum propositions \citep{heunen2009bohrification, nlab:bohr_topos}.

\section{Preliminaries}

\subsection*{Terminology}
\begin{definition}[Orthomodular Lattice]
\textbf{}\\An \textit{orthomodular lattice} is a structure \(L=(L,\wedge,\vee,\neg,0,1)\) consisting of a set \(L\) with two binary operations \(\wedge\) (meet) and \(\vee\) (join), a unary operation \(\neg\) (complement), and two distinguished elements \(0\) (the bottom element) and \(1\) (the top element), satisfying the following axioms:
\begin{enumerate}
    \item \textbf{Lattice Axioms:} \((L,\wedge,\vee)\) is a lattice, meaning that for all \(x,y,z\in L\):
    \begin{itemize}
        \item \(x\wedge y=y\wedge x\) \hfill (Commutativity of meet)
        \item \(x\vee y=y\vee x\) \hfill (Commutativity of join)
        \item \(x\wedge(y\wedge z)=(x\wedge y)\wedge z\) \hfill (Associativity of meet)
        \item \(x\vee(y\vee z)=(x\vee y)\vee z\) \hfill (Associativity of join)
        \item \(x\wedge(x\vee y)=x\) \hfill (Absorption law for meet)
        \item \(x\vee(x\wedge y)=x\) \hfill (Absorption law for join)
    \end{itemize}
    \item \textbf{Orthocomplementation:} For every element \(x\in L\), there exists a unique complement \(\neg x\in L\) such that:
    \begin{itemize}
        \item \(x\wedge\neg x=0\)
        \item \(x\vee\neg x=1\)
        \item \(\neg(\neg x)=x\)
    \end{itemize}
    \item \textbf{Orthomodularity:} For all \(x,y\in L\), if \(x\leq y\), then \(y=x\vee(y\wedge\neg x)\), where \(\leq\) denotes the partial order defined by \(x\leq y\) if and only if \(x\wedge y=x\).
\end{enumerate}
\end{definition}

\begin{definition}[Polymodal Logic]
\textbf{}\\\textit{Polymodal Logic} is an extension of modal logic that involves multiple modal operators, each of which may have different interpretations. These modal operators can represent various kinds of modalities such as temporal, epistemic, or deontic modalities. Formally, a polymodal logic with \(n\) modal operators is characterized by the set of formulas containing operators \(\square_1,\square_2,\dots,\square_n\), each corresponding to a different type of accessibility relation in the Kripke semantics. The structure of polymodal logic is detailed in \citep{goldblatt:logics_of_time_and_computation}.
\end{definition}

\begin{definition}[General Frames]
\textbf{}\\\textit{General Frames} extend Kripke frames by allowing for more general kinds of structures that can interpret modal logics. A general frame is defined as a tuple \((W, \mathcal{F}, \mathcal{R})\), where:
\begin{itemize}
    \item \(W\) is a non-empty set of worlds.
    \item \(\mathcal{F}\) is a field of subsets of \(W\) (i.e., a Boolean algebra of subsets of \(W\)).
    \item \(\mathcal{R}\subseteq\mathcal{F}\times\mathcal{F}\) is a relation between subsets of \(W\), which generalizes the accessibility relation of Kripke frames.
\end{itemize}
General frames are used to study completeness and other properties of modal logics, particularly when standard Kripke semantics may not be sufficient. Detailed discussions on general frames can be found in \citep{blackburn_rijke_venema:modal_logic}.
\end{definition}

\section{Maximal Boolean Sublattices}

In quantum mechanics, the mathematical structure of a Hilbert space \(H\) is often analyzed through the powerset \(\mathcal{P}(H)\) of its closed subspaces, which forms an orthomodular lattice \cite{birkhoff1936logic}. This orthomodularity reflects the non-classical logic that governs quantum systems \cite{piron1976foundations, Ptak2000}. However, within this non-classical framework, substructures exist that adhere to classical logic. These substructures, forming Boolean sublattices, correspond to classical snapshots within the quantum system. In these contexts, a set of mutually commuting observables defines a Boolean algebra, allowing us to view the quantum system from a classical perspective \cite{kochen1967problem, heunen2009bohrification}. These Boolean sublattices do not partition the entire orthomodular lattice; instead, they represent classical perspectives within the quantum framework.

\subsection{Maximal Chains and Boolean Sublattices}

We first introduce the notion of chains in a lattice:

\begin{definition}
\textbf{}\\A \textbf{maximal chain} in a lattice \(L\) is a totally ordered subset \(C\subset L\) such that for any two elements \(a,b\in C\), either \(a\leq b\) or \(b\leq a\). A chain is \textbf{maximal} if it includes both the least element (denoted by \(0\)) and the greatest element (denoted by \(1\)) of \(L\), and cannot be extended by adding any other elements from \(L\) without violating the total order.
\end{definition}

In the context of an orthomodular lattice, a maximal chain can be constructed by starting with the least element, iteratively adding orthogonal elements that preserve the chain's total order, and finally including the greatest element \cite{Ptak2000}.

\begin{definition}
\textbf{}\\A Boolean sublattice \(B\) of an orthomodular lattice \(L\) is called \textbf{maximal} if there is no other Boolean sublattice \(B'\subset L\) such that \(B\subset B'\). In simpler terms, \(B\) is maximal if it cannot be extended to a larger Boolean sublattice within \(L\) without losing its Boolean structure.
\end{definition}

Maximal Boolean sublattices represent the largest possible substructures that adhere to classical logic within the otherwise non-classical orthomodular framework of quantum logic. These sublattices correspond to sets of mutually commuting observables, which behave like classical variables, thus allowing us to view quantum systems from a classical perspective within the confines of quantum theory.

\subsection{Construction of a Maximal Boolean Sublattice}

To construct a maximal Boolean sublattice for an orthomodular lattice \(L\), the following procedure is typically used:

\begin{enumerate}
    \item \textbf{Start with a Maximal Chain \(C\) in \(L\)}: Begin with a maximal chain to ensure that the structure spans the lattice.
    
    \item \textbf{Define the Centralizer of \(C\)}: The centralizer, denoted \(\text{Cent}_L(C)\), is the set of all elements in \(L\) that commute with every element in \(C\). Formally,
    \[
    \text{Cent}_L(C)=\{x\in L\mid\forall a\in C,\,(x\wedge a)\vee(x^\perp\wedge a)=a\}.
    \]
    This step ensures that elements that commute with the entire chain \(C\) preserve the order and logical structure of \(C\).

    \item \textbf{Construct the Sublattice \(L_C\)}: Using the chain \(C\) and its centralizer \(\text{Cent}_L(C)\), construct a sublattice \(L_C\) by taking meets (greatest lower bounds) and joins (least upper bounds) of elements from \(C\) and \(\text{Cent}_L(C)\). Formally,
    \[
    L_C=\left\{x\in L\mid x=\bigvee_{i=1}^n\bigwedge_{j=1}^{m_i} y_{ij},\,y_{ij}\in C\cup\text{Cent}_L(C)\right\}.
    \]
    
    \item \textbf{Define the Maximal Boolean Sublattice \(B_C^{\text{max}}\)}: Finally, define the maximal Boolean sublattice \(B_C^{\text{max}}\) as the set of all elements in \(L_C\) and their complements within \(L\):
    \[
    B_C^{\text{max}}=\left\{x\in L\mid x=\bigvee_{i=1}^n\bigwedge_{j=1}^{m_i}z_{ij},\,z_{ij}\in L_C\cup\{y^\perp\mid y\in L_C\}\right\}.
    \]
\end{enumerate}

\begin{prop}
\(B_C^{\text{max}}\) is a maximal Boolean sublattice of \(L\).
\end{prop}

\begin{proof}
We will prove that \(B_C^{\text{max}}\) is a Boolean sublattice of \(L\) and that it is maximal, meaning there is no Boolean sublattice \(B'\) of \(L\) such that \(B_C^{\text{max}} \subset B'\).

\textbf{1. \(B_C^{\text{max}}\) is a Boolean sublattice:} We must show that \(B_C^{\text{max}}\) is closed under the operations of meet (\(\wedge\)), join (\(\vee\)), and complementation (\(x \mapsto x^\perp\)). Additionally, we must verify that \(B_C^{\text{max}}\) is distributive, a key property of Boolean algebras.

\begin{itemize}
    \item \textbf{Closure under Meet and Join:} By construction, every element \(x \in B_C^{\text{max}}\) can be expressed as a combination of elements from \(C\) and \(\text{Cent}_L(C)\). Since \(L_C\) is closed under meet and join, the operations within \(B_C^{\text{max}}\) will also be closed.
    
    \item \textbf{Closure under Complementation:} The complement of any element \(x\in B_C^{\text{max}}\) belongs to \(B_C^{\text{max}}\) because \(L_C\) is closed under complementation.
    
    \item \textbf{Distributivity:} To prove distributivity, we note that \(B_C^{\text{max}}\) is composed of elements that mutually commute. For any elements \(x,y,z\in B_C^{\text{max}}\), the distributive property requires that:
\[
x\wedge(y\vee z)=(x\wedge y)\vee(x\wedge z).
\]
Since all operations within \(B_C^{\text{max}}\) involve commuting elements from \(L_C\), the above equation holds, ensuring that \(B_C^{\text{max}}\) retains the Boolean structure, which inherently satisfies distributivity.

\end{itemize}

\textbf{2. \(B_C^{\text{max}}\) is Maximal:} Assume for contradiction that there exists a Boolean sublattice \(B' \subset L\) such that \(B_C^{\text{max}} \subset B'\) and \(B_C^{\text{max}} \neq B'\). This implies the existence of an element \(y \in B'\) such that \(y \notin B_C^{\text{max}}\). However, by the construction of \(B_C^{\text{max}}\), any element in \(B'\) that commutes with \(C\) must already be included in \(B_C^{\text{max}}\), leading to a contradiction. Thus, \(B_C^{\text{max}}\) is maximal.
\end{proof}

\subsubsection*{Connection to Quantum Mechanics: Basis/Sublattice Correspondence}
It is well-known that the orthomodular lattice of closed subspaces of a Hilbert space has a direct correspondence with the concept of bases in quantum mechanics \citep{piron1976foundations}. We verify this for the maximal Boolean sublattices as just constructed to further illustrate the connection between the abstract lattice-theoretic regime and more familiar Quantum mechanics.

\begin{theorem}[Basis/Sublattice Correspondence Theorem]
There is a one-to-one correspondence between the set of orthonormal bases of a finite-dimensional Hilbert space \(\mathcal{H}\) and the set of maximal Boolean sublattices in the orthomodular lattice \(\mathcal{L}\) of closed subspaces of \(\mathcal{H}\).
\end{theorem}

\begin{proof}
\textbf{}
   \begin{itemize}
       \item \(\mathbf{\Rightarrow}\): Let \(\{v_1,v_2,\ldots,v_n\}\) be an orthonormal basis of the Hilbert space \(\mathcal{H}\). We will show that this orthonormal basis generates a unique maximal Boolean sublattice within the orthomodular lattice \(\mathcal{L}\) of closed subspaces of \(\mathcal{H}\).

   \begin{enumerate}
       \item \textbf{Define the One-Dimensional Subspaces:} For each basis vector \(v_i\), define the one-dimensional subspace \(V_i=\mathrm{span}\{v_i\}\). Each subspace \(V_i\) is closed and corresponds to an element in the orthomodular lattice \(\mathcal{L}\).

       \item \textbf{Projection Operators:} The orthogonal projection operator onto \(V_i\) is denoted by \(P_i\). These projection operators have the following properties:
       \begin{itemize}
           \item \textbf{Orthogonality:} Since \(\{v_i\}\) is an orthonormal basis, the projection operators \(P_i\) satisfy \(P_iP_j=0\) for \(i\neq j\), meaning they project onto orthogonal subspaces.
           \item \textbf{Completeness:} The sum of these projection operators is the identity operator \(I=\sum_{i=1}^{n} P_i\), ensuring that the entire space \(\mathcal{H}\) is spanned by the subspaces \(V_i\).
       \end{itemize}

       \item \textbf{Generating the Maximal Boolean Sublattice:} The set of these projection operators \(\{P_1,P_2,\ldots,P_n\}\) forms a maximal chain in \(\mathcal{L}\). The sublattice \(B\) generated by this maximal chain consists of all possible finite joins (suprema) and meets (infima) of the projection operators \(P_i\) and their complements \(P_i^c=I-P_i\).

       \item \textbf{Isomorphism to a Boolean Algebra:}
           \begin{itemize}
               \item Since each projection operator \(P_i\) corresponds to a distinct one-dimensional subspace, and since these subspaces are mutually orthogonal and collectively span \(\mathcal{H}\), the structure generated by the set \(\{P_1, P_2, \ldots, P_n\}\) is isomorphic to the power set of the index set \(\{1, 2, \dots, n\}\), which is a Boolean algebra.
               \item \textbf{Maximality:} This Boolean algebra is maximal because adding any additional element to this structure would either violate the orthogonality of the subspaces or introduce an element that is not a join or meet of the projection operators and their complements. Therefore, the sublattice \(B\) is a maximal Boolean sublattice of \(\mathcal{L}\).
           \end{itemize}
       \end{enumerate}
       Thus, from any orthonormal basis \(\{v_1, v_2, \ldots, v_n\}\) of \(\mathcal{H}\), we can uniquely construct a maximal Boolean sublattice in the orthomodular lattice \(\mathcal{L}\).
       \item \(\mathbf{\Leftarrow}\): Now, we show the converse: starting from a maximal Boolean sublattice \(B\) in \(\mathcal{L}\), we will construct a unique orthonormal basis for \(\mathcal{H}\).

   \begin{enumerate}
       \item \textbf{Maximal Boolean Sublattice Properties:}
       \begin{itemize}
           \item Since \(B\) is a Boolean algebra and maximal, it contains a maximal chain of projection operators \(\{P_1,P_2,\ldots,P_n\}\) that satisfy the orthogonality condition \(P_iP_j=0\) for \(i\neq j\), and their sum is the identity operator \(I=\sum_{i=1}^{n} P_i\).
           \item The lattice \(B\) being Boolean implies that for every projection operator \(P_i\in B\), its complement \(P_i^c=I-P_i\) is also in \(B\), ensuring the lattice structure is closed under complementation.
       \end{itemize}

       \item \textbf{Correspondence to One-Dimensional Subspaces:} Each projection operator \(P_i\) in \(B\) corresponds to a one-dimensional subspace \(V_i=\mathrm{range}(P_i)\) of \(\mathcal{H}\). The range of \(P_i\) is spanned by a vector \(v_i\) in \(\mathcal{H}\), i.e., \(V_i=\mathrm{span}\{v_i\}\).

       \item \textbf{Orthonormality of Vectors:}
       \begin{itemize}
           \item The vectors \(v_i\) corresponding to each \(V_i\) are mutually orthogonal because the projection operators \(P_i\) are orthogonal (\(P_iP_j=0\) for \(i\neq j\)).
           \item \textbf{Normalization:} By normalizing these vectors (i.e., setting \(u_i=\frac{v_i}{\|v_i\|}\), where \(\|v_i\|\) is the norm of \(v_i\)), we obtain an orthonormal set \(\{u_1,u_2,\ldots,u_n\}\).
       \end{itemize}

       \item \textbf{Spanning the Hilbert Space:} The orthonormal set \(\{u_1, u_2, \ldots, u_n\}\) spans the entire Hilbert space \(\mathcal{H}\), as the sum of the projection operators is the identity operator. Hence, this set forms an orthonormal basis for \(\mathcal{H}\).
   \end{enumerate}
   Therefore, every maximal Boolean sublattice \(B\) in \(\mathcal{L}\) corresponds uniquely to an orthonormal basis of \(\mathcal{H}\).
   \end{itemize}
   By establishing both directions, we have shown that there is a one-to-one correspondence between the set of orthonormal bases of \(\mathcal{H}\) and the set of maximal Boolean sublattices in \(\mathcal{L}\). This completes the proof of the Basis/Sublattice Correspondence Theorem.
\end{proof}

\subsubsection*{Physical Interpretation} This theorem provides a deep connection between the abstract mathematical structures used in quantum logic and the concrete physical concept of a quantum state basis. In quantum mechanics, the choice of an orthonormal basis corresponds to choosing a set of commuting observables (represented by the projection operators) that can be simultaneously measured. The maximal Boolean sublattice represents the algebra of propositions about the system that can be definitively true or false given this choice of basis.

By establishing this correspondence, we gain insight into how classical logic (as represented by the Boolean sublattice) emerges within the quantum framework when restricting attention to a particular measurement context (the chosen orthonormal basis).

Building on the concepts of maximal Boolean sublattices and their connection to quantum mechanics, we now turn our attention to the broader mathematical structure that encompasses these sublattices. Specifically, we will explore lattice endomorphisms, which provide a framework for understanding the transformation of sublattices within an orthomodular lattice. 
\section{Lattice Endomorphisms}

In this section, we explore the relationship between unitary operators on Hilbert spaces and their corresponding unitary automorphisms within orthomodular lattices. This correspondence is crucial for understanding how different classical contexts, represented by maximal Boolean sublattices, are related through transformations that preserve logical consistency.

\subsection{Unitary Automorphisms}

\begin{definition}[Unitary Automorphism]
\textbf{}\\Let \(L\) be the Hilbert lattice of a finite-dimensional Hilbert space \(H\), and let \(B\) be the set of all maximal Boolean sublattices of \(L\). A unitary automorphism \(\Phi:L\rightarrow L\) is defined for sublattices \(B, B' \in B\) with associated sets of atoms \(A,A'\) by the function:
    \[
    \Phi_{B'}(x\in A)=\bigvee_{a'_i\in A'}(x\land a'_i).
    \]
    The automorphism \(\Phi\) must satisfy the following properties:
    \begin{itemize}
        \item \textbf{Linearity}: For any element \(x \in B\) that can be expressed as a join of atoms \(x=\bigvee_{i=1}^{k}a_i\), \(\Phi\) satisfies:
        \[
        \Phi(x)=\bigvee_{i=1}^{k}\Phi(a_i).
        \]
        \item \textbf{Preservation of Orthogonality}: For any orthogonal elements \(a, b\in B\), the images \(\Phi(a)\) and \(\Phi(b)\) are orthogonal in \(B'\):
        \[
        \Phi(a)\land\Phi(b)=0.
        \]
        \item \textbf{Preservation of Complements}: For any element \(x\in B\), the complement of \(x\) is preserved under \(\Phi\):
        \[
        \Phi(x^\perp)=(\Phi(x))^\perp.
        \]
        \item \textbf{Preservation of Order}: For any elements \(x,y\in L\) with \(x\leq y\), the automorphism \(\Phi\) preserves this order:
        \[
        \Phi(x)\leq\Phi(y).
        \]
    \end{itemize}
\end{definition}

These properties ensure that \(\Phi\) maintains the structure of the lattice, mirroring the behavior of unitary operators on a Hilbert space.

\subsection{Group Structure of Unitary Automorphisms}

\begin{prop}
    The set of unitary automorphisms \(\text{UA}(L)\) forms a group under composition.
\end{prop}

\begin{proof}
    To show that \(\text{UA}(L)\) forms a group under composition, we need to verify the group axioms: closure, associativity, identity, and inverses.
    \begin{itemize}
        \item \textbf{Closure}: Given two unitary automorphisms \(\Phi_{B'}\) and \(\Phi_{B''}\), their composition \(\Phi_{B''}\circ\Phi_{B'}\) is defined as:
        \[
        (\Phi_{B''}\circ\Phi_{B'})(x)=\Phi_{B''}(\Phi_{B'}(x)).
        \]
        Since each automorphism distributes over the join operation, the composition is also a unitary automorphism, satisfying closure.
        
        \item \textbf{Associativity}: The associativity of the composition follows directly from the associativity of function composition.
        
        \item \textbf{Identity}: The identity automorphism \(\Phi_{\text{id}}\) is defined as:
        \[
        \Phi_{\text{id}}(x)=x\quad\text{for all }x\in A.
        \]
        This automorphism trivially satisfies all the required properties, serving as the identity element.
        
        \item \textbf{Inverses}: For a unitary automorphism \(\Phi_{B'}\), the inverse \(\Phi_{B'}^{-1}\) must satisfy:
        \[
        \Phi_{B'}^{-1}(\Phi_{B'}(x))=x\quad\text{for all }x\in A.
        \]
        The existence of such an inverse follows from the bijective nature of \(\Phi_{B'}\), ensuring it can map atoms of \(B'\) back to atoms of \(B\).
    \end{itemize}
    Therefore, \(\text{UA}(L)\) forms a group under composition.
\end{proof}

\subsection{Isomorphism between \(\text{UA}(L)\) and \(U(H)\)}

\begin{theorem}
    The groups \(\text{UA}(L)\) and \(U(H)\) are isomorphic, where \(L\) is a Hilbert lattice and \(H\) is the corresponding Hilbert space.
\end{theorem}

\begin{proof}
    We define a mapping \(f:U(H)\rightarrow\text{UA}(L)\) as follows:
    \[
    f(U)=\Phi_U,
    \]
    where \(\Phi_U\) is the unitary automorphism induced by the unitary operator \(U\in U(H)\). This mapping preserves composition and satisfies both injectivity and surjectivity, establishing a group isomorphism between \(\text{UA}(L)\) and \(U(H)\).
\end{proof}

\section{Quantum Contextual Topos}

In this section, we define the Quantum Contextual Topos (QCT) and its associated topology. The QCT framework integrates topological and logical structures to model the contextual dependencies of quantum propositions. The construction of the QCT builds on the concept of a Quantum Frame, which extends traditional frames in modal logic to accommodate the unique requirements of quantum systems. By grounding these frames within a Stone space topology, we ensure that the logical model aligns with classical logic in suitable contexts while faithfully representing the non-classical features of quantum mechanics.

The QCT framework involves the development of key concepts such as basic open sets, covering families, and contextual topologies. These elements allow us to map logical propositions to specific topological structures, establishing a topology that mirrors the logical structure of quantum propositions. Central to this construction is the notion of contextuality, which is captured by the Quantum Frame and its associated topological features. Each context within an orthomodular lattice corresponds to a classical perspective, encapsulating a set of propositions that are consistent within a particular Boolean sublattice. This structure reflects the dynamic and non-commutative nature of quantum observables.

By developing the QCT, we customize such a topos-theoretic "universe" where the interplay between quantum logic and topology can be systematically explored. The internal logic of the QCT is constructed to align with the logical structure of the Quantum Frame, providing a unified approach to contextual quantum logic. Our framework harmonizes the logical and topological aspects of quantum reasoning, offering new insights into the underlying structures of quantum mechanics.

This section is organized as follows: We begin by defining the key concepts related to contexts and propositions within an orthomodular lattice, which are essential for establishing the underlying topology and logical framework. We then introduce the Quantum Frame, detailing its role in the overall structure of the QCT. Finally, we discuss the topological aspects of the QCT, including the construction of basic open sets and clopen sets, before concluding with a discussion on the dual algebra and its relevance to the QCT.

\begin{definition}[Logical Consequences]
\textbf{}\\Given a context \(c=(B,P)\) in \(\mathcal{C}(L)\), the logical consequences of \(c\), denoted \(\text{Th}(c)\), are the propositions \(q \in B\) entailed by \(P\). Formally,
\[
\text{Th}(c)=\{q \in B \mid P \models q\},
\]
where \(\models\) denotes the entailment relation in the Boolean algebra \(B\). This set reflects the deterministic nature of truth in classical logic.
\end{definition}

\begin{definition}[Compatibility and Filters]
\textbf{}\\A proposition \(q\) in \(B\) is compatible with a set of propositions \(P\) if \(q\wedge p\neq0_B\) for every \(p\in P\), where \(0_B\) is the bottom element of \(B\). The filter of compatible propositions \(F_c\) for a context \(c=(B,P)\) is the set of all propositions in \(B\) that are compatible with \(P\):
\[
F_c=\{q\in B\mid q\text{ is compatible with }P\}.
\]
This filter represents the maximal set of propositions within \(B\) that can be consistently true with \(P\), reflecting the classical logic embedded within the quantum context.
\end{definition}

\begin{definition}[Basis Extension]
\textbf{}\\Given a context \(c=(B,P)\) in \(\mathcal{C}(L)\), a basis extension to another basis \(B'\) identifies contexts in \(B'\) that are logically consistent with \(c\) via a unitary automorphism. This is formalized by the function
\[
E:\mathcal{C}(L)\times\mathcal{B}(L)\to\mathcal{P}(\mathcal{C}(L)),
\]
where \(\mathcal{B}(L)\) denotes the set of all bases in \(L\). For a given context \(c=(B,P)\) and a basis \(B'\), \(E(c,B')\) yields the set of all contexts \(c'=(B',P')\) such that:
\begin{itemize}
    \item There exists a unitary automorphism \(\Phi_{B\to B'} \in \text{UA}(L)\) that maps propositions from \(B\) to \(B'\).
    \item \(\text{Th}(c')\subseteq\text{Th}(c)\) under the automorphism \(\Phi\).
\end{itemize}
\end{definition}

\begin{definition}[Filter of Compatible Contexts]
\textbf{}\\The filter of compatible contexts \(F^\text{max}_c\) for a context \(c=(B,P)\) in \(\mathcal{C}(L)\) is the collection of all filters \(F_{c'}\) corresponding to contexts \(c'\) within the set \(E(c,B')\) for any basis \(B'\in\mathcal{B}(L)\):
\[
F^\text{max}_c=\{F_{c'}\mid c'\in E(c,B')\text{ for some }B'\in\mathcal{B}(L)\}.
\]
This filter extends compatibility to encompass all possible contexts that are logically consistent with \(c\).
\end{definition}

\subsection{Context Topology}
\begin{definition}[Basic Open Set]\label{bos}
\textbf{}\\For a context \(c=(B,P)\) in \(\mathcal{C}(\mathcal{L})\), the \textbf{basic open set} \(U_c\) is defined as:
\[
U_c=\{c'=(B',P')\in\mathcal{C}(\mathcal{L})\mid B'=B\text{ and }Th(c')\subseteq Th(c)\},
\]
where \(Th(c)\) is the set of all propositions in \(B\) that are logically entailed by the propositions in \(P\). The basic open set \(U_c\) represents the set of all contexts within the same Boolean sublattice \(B\) that share a consistent logical structure with \(c\).
\end{definition}

\begin{lemma}\label{une}
The set \(U_c\) is non-empty.
\end{lemma}
\begin{proof}
Consider the context \(c=(B,P)\) itself. By definition, \(Th(c)\) includes all propositions in \(B\) that are logically entailed by \(P\). Since \(c\) shares the same basis \(B\) with itself and \(Th(c)\subseteq Th(c)\) (as a set is always a subset of itself), \(c\) is an element of \(U_c\). Moreover, any context \(c'=(B, P')\) where \(P'\) generates a theory \(Th(c')\) that is a subset of \(Th(c)\) is included in \(U_c\). This is due to the fact that \(Th(c')\subseteq Th(c)\) implies all propositions in \(P'\) are consistent with and do not exceed the logical implications of \(P\). Therefore, \(U_c\) is non-empty as it always contains at least the context \(c\) itself.
\end{proof}

\begin{definition}[Immediately Accessible Contexts]\label{iac}
\textbf{}\\For a context \(c=(B,P)\) in \(\mathcal{C}(\mathcal{L})\), the \textbf{immediately accessible set} \(A_c\) is defined as the set of all contexts \(c'=(B',P')\) such that:
\begin{itemize}
    \item If \(B'=B\), then \(Th(c')\subseteq F_c\), meaning that the logical consequences of \(c'\) are compatible with those of \(c\).
    \item If \(B'\neq B\), then \(c'\in F^{max}_c\), meaning the contexts in \(B'\) that are logically consistent with \(c\) under basis extension are accessible, corresponding to potential "truths" in a different classical perspective.
\end{itemize}
The set \(A_c\) captures the notion of immediate logical reachability within and across bases, which is fundamental for understanding the dynamical aspect of logical structures in quantum contexts.
\end{definition}

\begin{definition}[Basis-Dependent Accessibility Chain]\label{accchain}
\textbf{}\\The \textbf{basis-dependent accessibility chain} \(\mathcal{A}^B_c\) for a context \(c=(B,P)\) is a sequence of sets \((\mathcal{A}_0^B, \mathcal{A}_1^B, \mathcal{A}_2^B, \ldots)\), where each \(\mathcal{A}_n^B\) is defined iteratively as follows:
\begin{enumerate}
    \item \(\mathcal{A}_0^B=\{U_c\};\)
    \item For \(n\geq 1\),
    \[
    \mathcal{A}_n^B=\mathcal{A}_{n-1}^B\cup\bigcup_{c'\in S_{n-1}^B}\{U_{c''}\mid c''=(B'',P'')\in\mathcal{C}(\mathcal{L}),B''=B,U_{c''}\in A_{c'}\},
    \]
    where \(A_{c'}\) is the immediately accessible set (as defined in Definition \ref{iac}) associated with \(c'\), and \(S_{n-1}^B\) is the set of all contexts \(c'\) such that \(U_{c'} \in \mathcal{A}_{n-1}^B\).
\end{enumerate}
Each set \(\mathcal{A}_n^B\) in the sequence expands the previous set \(\mathcal{A}_{n-1}^B\) by including additional basic open sets corresponding to contexts within the same basis \(B\). The construction of this chain is pivotal for understanding how logical propositions can be traced through various accessible contexts, ultimately supporting the topological structure that will underlie the Quantum Contextual Topos.
\end{definition}

\begin{definition}[Stone Space]
\textbf{}\\
A \textbf{Stone space} is a topological space that is:
\begin{enumerate}
    \item \textbf{Compact:} Every open cover has a finite subcover.
    \item \textbf{Hausdorff:} Any two distinct points have disjoint neighborhoods.
    \item \textbf{Zero-dimensional:} The space has a base of clopen sets (sets that are both open and closed).
\end{enumerate}
\end{definition}

Stone spaces are the topological spaces that arise from the duality between Boolean algebras and topological spaces (as captured by Stone's Representation Theorem). This theorem states that every Boolean algebra is isomorphic to the algebra of clopen sets of some Stone space, and conversely, the clopen sets of any Stone space form a Boolean algebra. \citep{stonereptheorem}.

\begin{cor}
\((\mathcal{C}(\mathcal{L}),\mathcal{J})\) is compact.
\end{cor}
\begin{proof}
Since \(\mathcal{L}\) is a finite orthomodular lattice, there are only finitely many maximal Boolean sublattices \(B\) in \(\mathcal{L}\). Let \(\{B_1,B_2,\ldots,B_n\}\) be the set of all distinct bases for contexts in \(\mathcal{C}(\mathcal{L})\). Given an open cover \(\mathcal{U}\) of \(\mathcal{C}(\mathcal{L})\), for each basis \(B_i\), there exists at least one basic open set \(U_{c_i}\in\mathcal{U}\) that covers all contexts with the basis \(B_i\). This is because the open cover \(\mathcal{U}\) must contain open sets that encompass all contexts for each basis \(B_i\). We can select one such \(U_{c_i}\) for each \(B_i\), where \(c_i=(B_i, P_i)\) for some \(P_i\) (the specific choice of \(P_i\) is irrelevant for covering purposes).
\\\\We claim that the collection \(\{U_{c_1},U_{c_2},\ldots,U_{c_n}\}\) forms a finite subcover of \(\mathcal{U}\) that covers \(\mathcal{C}(\mathcal{L})\). To see why, consider any context \(c=(B,P)\) in \(\mathcal{C}(\mathcal{L})\). The basis \(B\) must be one of the \(B_i\), say \(B=B_k\). The basic open set \(U_{c_k}\) associated with the context \(c_k=(B_k, P_k)\) covers all contexts with the basis \(B_k\), including the context \(c\). Therefore, every context in \(\mathcal{C}(\mathcal{L})\) is covered by at least one of the sets \(U_{c_1}, U_{c_2}, \ldots, U_{c_n}\), and thus this collection forms a finite subcover of \(\mathcal{U}\). Hence, \((\mathcal{C}(\mathcal{L}), \mathcal{J})\) is compact, as every open cover has a finite subcover.
\end{proof}
\begin{cor}
    \((\mathcal{C}(\mathcal{L}),\mathcal{J})\) satisfies point-closedness for all clopen sets.
\end{cor}
\begin{proof}
Let \(S\) be a clopen set in \((\mathcal{C}(\mathcal{L}),\mathcal{J})\) and let \(c \in S\) be an arbitrary context in the set. Since \(S\) is clopen, it can be expressed as a finite Boolean combination of basic open sets. Without loss of generality, let's consider the case where \(S\) is expressed as a finite union of finite intersections of basic open sets:
\[S=\bigcup_{k=1}^n \bigcap_{j=1}^{m_k} U_{c_{kj}}.\]
where each \(U_{c_{kj}}\) is a basic open set corresponding to a context \(c_{kj}\).
\\\\Since \(c \in S\), there exists an index \(k'\) such that \(c \in \bigcap_{j=1}^{m_{k'}} U_{c_{k'j}}\). Define the set \(V=\bigcap_{j=1}^{m_{k'}} U_{c_{k'j}}\). By construction, \(V\) is an intersection of basic open sets, and \(c \in V \subseteq S\).
\\\\Now, we need to show that there exists a clopen set \(W\) such that \(c \in W \subseteq V\). Since \(c\) belongs to the intersection \(V\), it must belong to each of the intersecting basic open sets \(U_{c_{k'j}}\). We can take \(W\) to be one of these basic open sets that contains \(c\), say \(W=U_{c_{k'1}}\). By definition, basic open sets are clopen, so \(W\) is a clopen set satisfying \(c \in W \subseteq V \subseteq S\), which demonstrates the point-closedness condition for the topological space of contexts.
\end{proof}

\begin{cor}
    \((\mathcal{C}(\mathcal{L}),\mathcal{J})\) is zero-dimensional.
\end{cor}

\begin{proof}
    To prove that \(\mathcal{C}(\mathcal{L})\) is zero-dimensional, we need to show that the topology on \(\mathcal{C}(\mathcal{L})\) has a base consisting entirely of clopen sets. 
    \\\\Let \(c=(B,P)\) be a context in \(\mathcal{C}(\mathcal{L})\), where \(B\subset\mathcal{L}\) is a maximal Boolean sublattice of the orthomodular lattice \(\mathcal{L}\), and \(P\) is a subset of propositions in \(B\). Recall, the basic open set \(U_c\) associated with \(c\) is defined as:
\[
U_c=\{ c'=(B',P') \in \mathcal{C}(\mathcal{L}) \mid B'=B \text{ and } \text{Th}(c') \subseteq \text{Th}(c) \},
\]
where \(\text{Th}(c)\) is the set of all propositions in \(B\) that are logically entailed by the propositions in \(P\).
\\\\By definition, \(U_c\) is open because it is a basic open set in the topology on \(\mathcal{C}(\mathcal{L})\). Next, consider the complement \(U_c^c\) of \(U_c\). This set consists of all contexts \(c'\) such that either \(B'\neq B\) or \(\text{Th}(c') \not\subseteq \text{Th}(c)\). The set \(U_c^c\) can be written as a union of basic open sets corresponding to contexts where \(\text{Th}(c')\) does not entail \(\text{Th}(c)\). Since the union of open sets is open, \(U_c^c\) is open, implying that \(U_c\) is closed. Since \(U_c\) is both open and closed, it is clopen.
\\\\Since the topology on \(\mathcal{C}(\mathcal{L})\) is generated by basic open sets \(U_c\), and we have shown that each basic open set is clopen, the collection of these clopen sets forms a base for the topology on \(\mathcal{C}(\mathcal{L})\).
Finally, we verify zero-dimensionality. For any open set \(U \subset \mathcal{C}(\mathcal{L})\) and any point \(c \in U\), there exists a basic open set \(U_c\) such that \(c\in U_c\subseteq U\). Since \(U_c\) is clopen, it follows that the topology on \(\mathcal{C}(\mathcal{L})\) has a base consisting entirely of clopen sets.
\end{proof}

\begin{lemma}\label{stone}
    \(\mathcal{C}(\mathcal{L})\) is a Stone space.
\end{lemma}

\begin{proof}
    This follows directly from \textbf{corollaries 1, 2, and 3}.
\end{proof}

\subsection{Topos}

This subsection introduces key concepts related to quantum frames, dual algebras, and double-dual frames, emphasizing their role in capturing the logical structure of quantum frames. We conclude the overall section by formally defining the QCT.

\subsubsection{Quantum Frame}

\begin{definition}[Quantum Frame]
\textbf{}\\A \textbf{Quantum Frame} is a polymodal general frame \(F(L)=(C,\Gamma,\{R^{B_i}_p\},V)\) where:
\begin{itemize}
    \item \(C=C(L)\) is the set of contexts of an orthomodular lattice \(L\). For brevity, we denote it as \(C\).
    \item The accessibility relations \(\{R^{B_i}_p\}\) are defined as follows: For each proposition \(p\in L\) and each basis \(B_i\in B(L)\) (a maximal Boolean sublattice of \(L\)), \(R^{B_i}_p\subseteq C\times C\) is a relation on contexts such that for any two contexts \(c=(B,P)\) and \(c'=(B',P')\), \((c,c')\in R^{B_i}_p\) if and only if:
    \begin{enumerate}
        \item \(p\in Th(c)\), implying there exists some \(c''\in U_c\) such that \(c''=(B''=B,P'')\);
        \item \(p\) is the maximal element in \(Th(c'')\);
        \item \(c'\in A_1^{B}(c'')\).
    \end{enumerate}
    \item \(V\subseteq\mathcal{P}(C)\) is the set of clopen subsets of \(C\), closed under Boolean operations and modal operators. For a formula \(\omega\in V\), the modal operator \(\square^i_p\) is defined as:
    \[
    \square^i_p\omega=\{c\in C\mid\forall(c,c')\in R^{B_i}_p,\,c'\in\omega\}.
    \]
\end{itemize}
\end{definition}

\subsubsection{Quantum Formulae in \(V\)}

\begin{definition}[Quantum Formulae in \(V\)]
\textbf{}\\A quantum formula \(\varphi\) in \(V\) is defined recursively as follows:
\begin{itemize}
    \item A clopen set \(\omega\in V\) is a quantum formula.
    \item If \(\varphi\) is a quantum formula, then the negation \(\neg\varphi\) is also a quantum formula.
    \item If \(\varphi_1\) and \(\varphi_2\) are quantum formulae, then the conjunction \(\varphi_1 \wedge \varphi_2\) and the disjunction \(\varphi_1 \vee \varphi_2\) are also quantum formulae.
    \item If \(\varphi\) is a quantum formula and \(B_i\in B(L)\) is a basis, then the necessity \(\square^{B_i}_p\varphi\) is also a quantum formula.
    \item If \(\varphi_1\) and \(\varphi_2\) are quantum formulae, then the implication \(\varphi_1\rightarrow\varphi_2\) is also a quantum formula, corresponding to set inclusion.
\end{itemize}
\end{definition}

\begin{prop}\label{landmod}
    \(\square^{B_i}_p\omega\land\square^{B_i}_p\omega'\Leftrightarrow\square^{B_i}_p(\omega\land\omega')\) holds in a quantum frame, where \(\omega,\omega'\in V\).
\end{prop}
\begin{proof}\textbf{}
    \begin{itemize}
        \item \(\Rightarrow\): The LHS applies the modal operator to the intersection of two sets:
        \[
        \square^{B_i}_p(\omega \cap \omega')=\{c\in\mathcal{C}\mid\forall c'\in\mathcal{C},(c,c')\in R^{B_i}_p\implies c'\in\omega\cap\omega'\}.
        \]
        If \(c\) is in \(\square^{B_i}_p(\omega\cap\omega')\), then for all \(c'\) accessible from \(c\), \(c'\) must be in both \(\omega\) and \(\omega'\). This directly implies that \(c\) satisfies both \(\square^{B_i}_p\omega\) and \(\square^{B_i}_p\omega'\), and therefore \(c\) is in their intersection.

        \item \(\Leftarrow\): The RHS takes the intersection of the modal transformations of two sets:
        \begin{equation*}
\begin{aligned}
    \square^{B_i}_p\omega\cap\square^{B_i}_p\omega'&=\{c\in\mathcal{C}\mid\forall c'\in\mathcal{C},(c,c')\in R^{B_i}_p\implies c'\in\omega\} \\
    &\quad\cap\{c\in\mathcal{C}\mid\forall c'\in\mathcal{C},(c,c')\in R^{B_i}_p\implies c'\in\omega'\}.
\end{aligned}
\end{equation*}

        If \(c\) is in \(\square^{B_i}_p\omega\cap\square^{B_i}_p\omega'\), then for all \(c'\) accessible from \(c\) under \(R^{B_i}_p\), \(c'\) must be in both \(\omega\) and \(\omega'\). Therefore, \(c'\) must be in \(\omega\cap\omega'\), implying \(c\) is in \(\square^{B_i}_p(\omega\cap\omega')\).
    \end{itemize}
    This completes the proof.
\end{proof}

\subsubsection*{Clopen Sets}

\begin{definition}[Clopen Sets]
\textbf{}\\In the context of a quantum frame \(F(L)=(C,\{R_i\},V)\), a clopen set is a subset of \(C(L)\) that is both open and closed under the topology induced by the lattice operations and the modal structure \citep{johnstone1982stone}.

The space \(C(L)\) consists of all contexts, where each context corresponds to a unique combination of a maximal Boolean sublattice \(B\) of \(L\) and a set \(P\) of propositions considered \textit{true} within that sublattice. The topology on \(C(L)\) is generated by the basic open sets, which are defined as:
\[
U_c=\{c'\in C(L)\mid Th(c')\subseteq Th(c)\}.
\]
where \(c=(B,P)\) is a context, and \(Th(c)\) denotes the theory generated by \(P\) within \(B\).

Clopen sets are constructed using finite unions, intersections, and complements of basic open sets \citep{simmons1963introduction}. These operations ensure that the clopen nature is preserved due to the finite and closed topology of \(C(L)\).
\end{definition}

\begin{theorem}
    In a quantum frame \(F(L)=(C,\{R_i\},V)\), there is a direct correspondence between clopen sets in \(V\) and propositions in the Boolean sublattices of \(L\).
\end{theorem}
\begin{proof}
    Let \(F(L)=(C,\{R_i\},V)\) be a quantum frame where \(V\) is the set of clopen sets, defined as subsets of \(C\) that are both open and closed under the topology generated by basic open sets \(U_c=\{c'\in C\mid Th(c')\subseteq Th(c)\}\).
    \begin{enumerate}
        \item \textbf{Clopen Sets as Propositional Sets}: By definition, the basic open set \(U_c\) corresponds to the set of all contexts \(c'=(B',P')\) such that \(Th(c')\subseteq Th(c)\). This implies that \(U_c\) is entirely determined by the propositions \(Th(c)\) that are true in the Boolean sublattice \(B\).
        
        A clopen set \(\omega\in V\) can be expressed as a finite union, intersection, or complement of such basic open sets \(U_c\). Therefore, each clopen set in \(V\) is ultimately defined by the collection of contexts in \(C\), which are in turn defined by the propositions in their respective Boolean sublattices.
        \item \textbf{Correspondence between Clopen Sets and Propositions}: Consider a maximal Boolean sublattice \(B\subset L\). The propositions within \(B\) determine the theories \(Th(c)\) for contexts \(c=(B,P)\) associated with \(B\). Since clopen sets in \(V\) are constructed from basic open sets corresponding to these contexts, the clopen sets are directly associated with the propositions in \(B\).

        To establish a bijective correspondence, we construct the following mapping:
        \begin{align*}
            \phi: \{\text{Propositions in } B\} &\longrightarrow \{\text{Clopen Sets in }V\} \\
            p &\longmapsto\{c'=(B',P')\in C\mid p\in Th(c')\}
        \end{align*}
        
        This map \(\phi\) is well-defined because:
        \begin{itemize}
            \item Each proposition \(p\) in \(B\) corresponds to the set of contexts where \(p\) is true, which forms a clopen set in \(V\).
            \item Each clopen set in \(V\) can be represented by a finite collection of such propositions, given the construction from basic open sets.
        \end{itemize}
        \item \textbf{Bijectivity}: 
        \begin{itemize}
            \item \textbf{Injectivity}: Assume \(\phi(p)=\phi(q)\) for propositions \(p,q\in B\). This implies that \(p\) and \(q\) are true in exactly the same set of contexts. Since the Boolean algebra \(B\) satisfies classical logic, \(p=q\) follows.
            \item \textbf{Surjectivity}: For any clopen set \(\omega \in V\), \(\omega\) is a finite union, intersection, or complement of basic open sets \(U_c\), each of which is associated with a proposition in some Boolean sublattice \(B\). Therefore, \(\omega\) can be mapped back to a unique proposition or set of propositions in \(B\).
        \end{itemize}
    \end{enumerate}

Thus, \(\phi\) is a bijection, establishing a direct correspondence between clopen sets in \(V\) and propositions in the Boolean sublattices of \(L\).

\end{proof}

\subsubsection{Dual Algebra and Dual Frame}

\begin{definition}[Dual Algebra/Dual Frame]
\textbf{}\\Given a polymodal general frame \(G=(W,R_1,R_2,\dots,R_n)\), its dual algebra is a Boolean algebra with \(n\) modal operators, forming a polymodal algebra \citep{nlab:algebraic_model_for_modal_logics}. The dual algebra \(G^+\) is defined as \(\langle X,\cap,\cup,-,\square_1,\square_2, \dots,\square_n\rangle\), where:
\begin{itemize}
    \item \(X\) is the set of clopen subsets of \(W\) \citep{johnstone1982stone}.
    \item \(\cap,\cup,-\) are the standard Boolean operations on sets.
    \item \(\square_i\) is the modal operator corresponding to the relation \(R_i\), defined as \(\square_i U=\{w\in W\mid\forall v(wR_iv\implies v\in U)\}\) for each \(i=1,2,\dots,n\) \citep{blackburn_rijke_venema:modal_logic}.
\end{itemize}

The dual frame of the dual algebra \(G^+\) is a polymodal general frame \(G^+=(F,R'_1,R'_2,\dots,R'_n,V)\), where:
\begin{itemize}
    \item \(F\) is the set of ultrafilters of the dual algebra \(G^+\) \citep{johnstone1982stone}.
    \item \(R'_i\) is the relation on ultrafilters corresponding to the modal operator \(\square_i\), defined as \((f,f')\in R'_i\) if and only if \(\forall U\in X(\square_iU\in f\implies U\in f')\) for each \(i=1,2,\dots,n\) \citep{nlab:algebraic_model_for_modal_logics, blackburn_rijke_venema:modal_logic, chagrov2000}.
    \item \(V\) is the set of clopen subsets of \(F\), corresponding to the elements of the dual algebra \(G^+\) \citep{johnstone1982stone}.
\end{itemize}
\end{definition}
\begin{definition}[Quantum Contextual Topos]
    \textbf{}\\The \textbf{Quantum Contextual Topos} (QCT) is formally defined as the category of Quantum Frames over the category of orthomodular lattices.
\end{definition}

\section{Internal Logic of QCT}

In this section, we delve into the internal logic of the Quantum Contextual Topos. The discussion begins by establishing the necessary theoretical groundwork, including the adaptation of Stone’s Representation Theorem within the QCT context. We then proceed to construct the double-dual frame, an essential component that guarantees the logical consistency and bijective correspondence between quantum contexts and their classical counterparts. By proving the preservation of logical relations and valuations through the isomorphism \(\Theta\), we solidify the argument that the QCT framework  aligns with the classical logical approach to quantum reasoning.

\begin{definition}[Stone's Representation Theorem]
\textbf{}\\\textbf{Stone's Representation Theorem} states that for any Boolean algebra \(\mathcal{B}\), there exists a Stone space \(\mathcal{X}\) such that \(\mathcal{B}\) is isomorphic to the Boolean algebra of clopen sets of \(\mathcal{X}\). Conversely, for any Stone space \(\mathcal{X}\), the collection of its clopen sets forms a Boolean algebra that is isomorphic to some Boolean algebra \(\mathcal{B}\).
\end{definition}
In the Quantum Contextual Topos, the space of contexts \((\mathcal{C}(\mathcal{L}),\mathcal{J})\) of an orthomodular lattice \(\mathcal{L}\) forms a Stone space. By introducing modal operators, however, we introduce additional complexity. We will adapt the theorem to account for the slightly different structure.

\begin{theorem}
    The internal logic of the topos is propositional polymodal logic.
\end{theorem}
    To prove the theorem, we must show that for any quantum frame \(\mathcal{F}(\mathcal{L})\) associated with an orthomodular lattice \(\mathcal{L}\), the internal logic of the Quantum Contextual Topos (QCT) can be identified with propositional polymodal logic. Specifically, this involves demonstrating that:
    \begin{enumerate}
        \item \textbf{Double-Dual Frame Identity}: The quantum frame \(\mathcal{F}(\mathcal{L})\) is isomorphic to its double-dual \(\mathcal{F}(\mathcal{L})^{++}\), thus ensuring that the logical structure is preserved under dualization.
        \item \textbf{Preservation of Logical Relations and Valuations}: The isomorphism \(\Theta\) preserves both the accessibility relations and the valuations of quantum propositions.
    \end{enumerate}
    The proof will proceed in three main steps:
    \begin{enumerate}
        \item \textbf{Constructing the Double-Dual Frame}: We need to define the double-dual frame \(\mathcal{F}(\mathcal{L})^{++}\) and show that it is structurally identical to the original frame \(\mathcal{F}(\mathcal{L})\). 
        \item \textbf{Defining the Isomorphism \(\Theta\)}: We will construct an isomorphism \(\Theta\) between the original frame \(\mathcal{F}(\mathcal{L})\) and its double-dual \(\mathcal{F}(\mathcal{L})^{++}\). This mapping must be shown to be bijective, and it must preserve the logical structure.
        \item \textbf{Verifying Preservation of Logical Relations}: We will prove that the isomorphism \(\Theta\) preserves the accessibility relations and the valuations within the logical system.
    \end{enumerate}
    \subsection*{Assumptions}
Before proceeding with the detailed proof, we outline the major assumptions that are implicitly guiding the proof. These assumptions are critical to ensure the validity and logical consistency of the arguments presented.

\begin{enumerate}
    \item \textbf{Orthomodular Lattice Structure:} We assume that the orthomodular lattice \(\mathcal{L}\) associated with the quantum frame \(\mathcal{F}(\mathcal{L})\) possesses the standard algebraic properties, including orthocomplementation, orthomodularity, and the existence of maximal Boolean sublattices. These properties are necessary for the proper definition of contexts and the construction of Boolean sublattices within the lattice.

    \item \textbf{Properties of the Topology on Contexts:} We assume that the topology on the set of contexts \(C(\mathcal{L})\) behaves classically, supporting the definition and manipulation of basic open sets and clopen sets. This includes the assumption that \(C(\mathcal{L})\) is compact and that clopen sets are closed under intersections and unions. This assumption is justified by \textbf{Lemma \ref{stone}}.

    \item \textbf{Closure Properties of Clopen Sets:} We assume that the set of clopen subsets \(V\) in \(\mathcal{F}(\mathcal{L})\) is rich enough to support all necessary logical operations, including those required by modal logic, and that it satisfies the closure properties necessary for the ultrafilter construction.

\end{enumerate}

    \subsection*{Proof}
    \begin{proof}
    Let \(\mathcal{F}(\mathcal{L})=(\mathcal{C},\{R_i^j\},V)\) be the quantum frame associated with an orthomodular lattice \(\mathcal{L}\). The dual algebra \(\mathcal{F}(\mathcal{L})^+\) of \(\mathcal{F}(\mathcal{L})\) is defined as \(\langle X,\cap,\cup,-,\square_i^j\rangle\), where \(X\) is the set of clopen subsets of \(\mathcal{C}\), and \(\square_i^j\) are the modal operators corresponding to the relations \(R_i^j\).
    \\\\The double-dual \(\mathcal{F}(\mathcal{L})^{++}\) is the dual frame of the dual algebra \(\mathcal{F}(\mathcal{L})^+\). It is a polymodal general frame \((F,\{R_i'^j\},V')\), where \(F\) is the set of ultrafilters of \(\mathcal{F}(\mathcal{L})^+\), \(R_i'^j\) are the relations on ultrafilters corresponding to the modal operators \(\square_i^j\), and \(V'\) is the set of clopen subsets of \(F\).
    \\\\We will prove the duality by constructing an isomorphism \(\Theta: \mathcal{F}(\mathcal{L}) \rightarrow \mathcal{F}(\mathcal{L})^{++}\) between the original quantum frame and its canonical double-dual. This mapping must be bijective between the set of worlds in both frames (contexts and ultrafilters, respectively) and must preserve the logical properties. Specifically, \(\Theta\) must be bijective, preserve accessibility relations, and preserve valuations. Once such a mapping is constructed and validated, we will have successfully demonstrated that any formula in the propositional polymodal logic of \(\mathcal{F}(\mathcal{L})^{++}\) is equivalent to a formula in \(\mathcal{F}(\mathcal{L})\) for any quantum frame.

    \subsection{Ultrafilter Construction for Contexts}
    To construct the isomorphism, we define \(\Theta\) such that each context \(c \in \mathcal{C}\) is assigned a filter on \(V\) in \(\mathcal{F}(\mathcal{L})^{++}\) (i.e., \(\Theta(c)\equiv F_c\)) as follows:

    \begin{definition}[Filter for a Context \(c\)]
    \textbf{}\\For a context \(c=(B,P)\) in \(\mathcal{C}\):
        \[
        F_c=\{\omega\in V\mid U_c\subseteq\omega\}\cup\bigcup_{B_i\in\mathcal{B}}\bigcup_{p\in P}M_{B_i,p}.
        \]
        where:
        \begin{itemize}
            \item \(U_c\) is the basic open set of context \(c\), encompassing all contexts sharing the same basis \(B\) with theories (logical consequences) that are a subset of or equal to \(Th(c)\).
            \item \(M_{B_i,p}\) represents the set of modal formulas satisfied at the context, covering all different modal operators in the logic, defined as:
            \[
            M_{B_i,p}=\{\square^{B_i}_p\omega\mid\omega\in V,\forall c'\in\mathcal{A}_1^{B_i}(c''),c'\in\omega\}.
            \]
        \end{itemize}
        Here, \(\square^{B_i}_p\omega\) denotes the application of the modal operator to \(\omega\) and \(\mathcal{A}_1^{B_i}\) is the basis-dependent accessibility chain (\textbf{def \ref{accchain}}) of depth \(1\).
    \end{definition}\textbf{}
    \begin{aside}
        \textbf{Aside:} The interaction between the modal operators and a context's propositions will be explicated: The modal operator \(\square^{B_i}_p\) represents the necessity of a formula \(\omega\) across all contexts \(c'\) accessible from \(c\) under the relation \(R^{B_i}_p\). Including \(M_{B_i,p}\) in the filter \(F_c\) ensures that \(F_c\) not only reflects the propositions directly associated with the context \(c\) but also those that must hold in all accessible contexts, capturing the modal structure of the quantum frame.
        
        Regarding the modal operator's impact on the structure of \(F_c\), the sets \(M_{B_i,p}\) expand \(F_c\) beyond the basic open sets \(U_c\) by incorporating necessary modal formulas, ensuring that \(F_c\) is closed under the operations defined by the modal operators. This inclusion maintains logical consistency across contexts and aligns \(F_c\) with the orthomodular structure of the lattice \(\mathcal{L}\).
    \end{aside}\textbf{}

    \begin{cor}
        \(F_c\) is a filter on \(V\).
    \end{cor}

    \begin{proof}
        We verify that \(F_c\) satisfies the conditions of non-emptiness, upward closure, and closure under intersections for all clopen sets \(\omega \in V\).
        \begin{enumerate}
            \item \textbf{Non-emptiness}: \(F_c\) is non-empty because it always contains the basic open set \(U_c\), which is guaranteed to be non-empty by \textbf{Lemma \ref{une}}.
            \item \textbf{Upward Closure}: \(F_c\) is upward closed if for every set \(\omega \in F_c\), any superset \(\omega' \in V\) of \(\omega\) (i.e., \(\omega \subseteq \omega'\)) also belongs to \(F_c\). \(F_c\) includes both the basic open set \(U_c\) and other sets determined by modal operations within the lattice of clopen sets \(V\).
            \begin{enumerate}
                \item \textbf{Case of Basic Open Sets}: Suppose \(\omega\in F_c\) such that \(U_c\subseteq\omega\). Let \(\omega'\in V\) be a clopen set such that \(\omega\subseteq\omega'\). By the definition of \(F_c\), since \(U_c\subseteq\omega\) and \(\omega\subseteq\omega'\), it follows by the transitivity of set inclusion that \(U_c\subseteq\omega'\). Thus, \(\omega'\) meets the criterion for inclusion in \(F_c\), as \(F_c\) includes all clopen sets in \(V\) that are supersets of \(U_c\).
                \item \textbf{Case of Modal Formulas}: Suppose \(\omega\in F_c\) due to a modal formula, i.e., \(\omega=\square^{B_i}_p\omega''\) for some \(\omega''\in V\). Let \(\omega'\in V\) be such that \(\omega\subseteq\omega'\). Since \(U_c\subseteq\omega\) by definition of the modal operator, and \(\omega'\) is a superset of \(\omega\), by the transitivity of set inclusion, \(U_c\subseteq\omega'\). Therefore, \(\omega'\) must belong to \(F_c\), as \(F_c\) includes all clopen sets in \(V\) that contain \(U_c\).
            \end{enumerate}
            \item \textbf{Closure under Intersections}: Consider two arbitrary clopen sets \(\omega\) and \(\omega'\) that are elements of \(F_c\). We aim to show that their intersection, \(\omega\cap\omega'\), also belongs to \(F_c\).
            \begin{itemize}
                \item \textbf{Case of Basic Open Sets}: Assume \(\omega,\omega'\in F_c\). By the definition of \(F_c\), we have \(U_c\subseteq\omega\) and \(U_c\subseteq\omega'\). Given that \(\omega\) and \(\omega'\) are clopen sets, their intersection \(\omega\cap\omega'\) is also a clopen set because the space \(V\) (comprised of clopen sets) is closed under finite intersections. Therefore, \(U_c\subseteq\omega\cap\omega'\), and \(\omega\cap\omega'\) meets the criterion for inclusion in \(F_c\).
                \item \textbf{Case of Modal Formulas}: Suppose \(\omega=\square^{B_i}_p\omega_0\) and \(\omega'=\square^{B_i}_p\omega_1\) are both contained in \(F_c\). This means that in all accessible contexts \(c'\) from \(c\), both \(\omega_0\) and \(\omega_1\) are satisfied, i.e., \(c'\in\omega_0\cap\omega_1\). Hence, \(\omega_0\cap\omega_1\in F_{c'}\). Therefore, \(\square^{B_i}_p(\omega_0\land\omega_1)\in F_c\). By (\textbf{Prop} \ref{landmod}), this is equivalent to \(\square^{B_i}_p\omega_0\land\square^{B_i}_p\omega_1\in F_c\).
            \end{itemize}
        \end{enumerate}
        These properties confirm that \(F_c\) meets all criteria of a filter on \(V\).
    \end{proof}

    \begin{lemma}
        \(F_c\) is an ultrafilter on \(V\).
    \end{lemma}

    \begin{proof}
        Since \(C(\mathcal{L})\) is a Stone space (\textbf{Lemma \ref{stone}}), \(F_c\) is automatically a maximal filter, which is by definition an ultrafilter. The maximality of the filter \(F_c\) follows directly from the properties of Stone spaces.
    \end{proof}
    
    Since \(F_c\) was constructed for an arbitrary context \(c\), we conclude that the mapping \(\Theta\) is well-defined, associating to any applicable context \(c\in\mathcal{C}\) in a quantum frame \(\mathcal{F}(\mathcal{L})\) an ultrafilter on \(V\).
    \\
    \begin{aside}
    \textbf{Aside:} We examine two edge cases to illustrate and verify that \(F_c\) is an ultrafilter in the context set \(c=(B,P)\). These cases are defined by minimal and maximal logical content in \(P\).
    \begin{itemize}
        \item \textbf{Case One: Minimal Logical Content in \(P\)}:
        Assume \(P\) is empty, indicating no propositions are considered true a priori in \(c\). In this scenario, \(Th(c)\) includes all tautologically true propositions in \(B\). Thus, \(U_c\) includes all contexts \(c'=(B',P')\) where \(B'=B\) and \(P'\) imposes no additional constraints beyond those universally true in \(B\). Therefore, \(F_c\) begins with a broad basis, containing all sets \(\omega\in V\) where \(U_c \subseteq\omega\). Since \(U_c\) is extensive, this condition implies \(F_c\) starts broadly but remains well-defined. Next, we verify the ultrafilter properties for \(F_c\).
        \begin{itemize}
            \item \textit{Non-emptiness}: \(F_c\) contains \(U_c\), thus it is non-empty.
            \item \textit{Upward Closure}: Any \(\omega\in V\) containing \(U_c\) is in \(F_c\), and any supersets of such \(\omega\) are also included.
            \item \textit{Closure under Intersections}: If \(\omega,\omega'\in F_c\), both containing \(U_c\), then \(\omega\cap\omega'\) also contains \(U_c\) and hence belongs to \(F_c\).
            \item \textit{Maximality}: For any \(\omega\in V\), because \(U_c\) is broad and inclusive, if \(\omega\notin F_c\), this implies \(\omega\) does not cover all possible valuations permitted by the empty \(P\); thus, its complement, \(\neg\omega\), which will cover these, must be in \(F_c\).
        \end{itemize}
        \item \textbf{Case Two: Maximal Logical Content in \(P\)}:
        Assume \(P\) for a context \(c=(B,P)\) is maximally restrictive, possibly containing contradictory or comprehensive sets of propositions within \(B\). With \(P\) maximally filled, \(Th(c)\) potentially includes a highly restrictive set of propositions, narrowing down \(U_c\) to very few contexts or even just \(c\) itself if \(P\) determines a unique valuation. Starting from possibly the singleton set \(U_c\), \(F_c\) includes all supersets of a potentially minimal \(U_c\), focusing the filter significantly. Next, we verify the ultrafilter properties for \(F_c\). \textit{Non-emptiness}, \textit{upward closure}, and \textit{closure under intersections} are maintained as before. For \textit{maximality}, given the minimal nature of \(U_c\), any \(\omega\in V\) either directly aligns with the stringent criteria set by \(P\) or it doesn't. If it doesn't, \(\neg\omega\), encompassing the necessary criteria, must be in \(F_c\).
    \end{itemize} 
    In both edge cases, \(F_c\) preserves its ultrafilter properties.
   \end{aside}\textbf{}
   \subsection{Bijectivity of \(\Theta\)}
   We now prove that \(\Theta:\mathcal{C}\rightarrow F\) is both injective and surjective.
   
   \begin{cor}
       \(\Theta\) is injective.
   \end{cor}
   
   \begin{subproof}
        Injectivity follows from the fact that each context \(c\in\mathcal{C}\) is uniquely associated with a distinct ultrafilter \(F_c\) in \(F\). Given the zero-dimensionality of the Stone space, different contexts correspond to disjoint clopen sets, ensuring that the ultrafilters are distinct. We will, however, proceed with a more detailed proof by contradiction.
        \\\\Assume, for contradiction, that there exist two distinct contexts \(c_1=(B_1,P_1)\) and \(c_2=(B_2,P_2)\) in \(\mathcal{C}\) such that \(c_1\neq c_2\), but \(\Theta(c_1)=\Theta(c_2)\). This would imply that the ultrafilter \(F_{c_1}\) associated with \(c_1\) is equal to the ultrafilter \(F_{c_2}\) associated with \(c_2\).
        \\\\To demonstrate a contradiction, we consider two cases: either the contexts \(c_1\) and \(c_2\) have different bases, or they share the same basis but differ in their set of propositions.
   
       \begin{itemize}
           \item \textbf{Case One (Different Bases)}: Suppose \(B_1\neq B_2\). The basic open sets \(U_{c_1}\) and \(U_{c_2}\) corresponding to \(c_1\) and \(c_2\) are distinct and disjoint because basic open sets corresponding to contexts from different bases are disjoint by construction. This disjointness reflects the orthomodular structure of the lattice, where non-commuting sets of propositions cannot coexist in the same context. Since ultrafilters are maximal filters that do not contain disjoint sets, \(F_{c_1}\) and \(F_{c_2}\) cannot be equal, implying that \(\Theta(c_1) \neq \Theta(c_2)\).

           \item \textbf{Case Two (Shared Basis)}: Suppose \(B_1=B_2=B\) and \(P_1\neq P_2\). We consider two subcases:
           \begin{itemize}
               \item \textbf{Subcase 2.1: \(P_1\not\subseteq P_2\) and \(P_2 \not\subseteq P_1\)}: Without loss of generality, assume there exists a proposition \(p\in P_1\) such that \(p\notin P_2\). Define \(\omega_p=\{c \in \mathcal{C} \mid p \in Th(c)\}\), the set of all contexts where \(p\) is a logical consequence. By construction, \(\omega_p\in V\) because \(V\) is closed under Boolean operations, and \(\omega_p\) represents a valid operation in the Boolean algebra of the lattice. Since \(U_{c_1}\subseteq\omega_p\), it follows that \(\omega_p\in F_{c_1}\). However, since \(p\notin P_2\), \(U_{c_2}\not\subseteq\omega_p\), which implies \(\omega_p\notin F_{c_2}\). Thus, \(F_{c_1}\neq F_{c_2}\), leading to the conclusion that \(\Theta(c_1)\neq\Theta(c_2)\).

               \item \textbf{Subcase 2.2: \(Th(c_1)\subseteq Th(c_2)\) or \(Th(c_2)\subseteq Th(c_1)\)}: Without loss of generality, assume \(Th(c_1)\subseteq Th(c_2)\). This implies \(Th(c_1)\subseteq\Theta(c_1)\) and \(Th(c_1)\subseteq\Theta(c_2)\). Since the theories of \(c_1\) and \(c_2\) are related by subset inclusion, the corresponding ultrafilters \(F_{c_1}\) and \(F_{c_2}\) cannot be equal unless \(P_1=P_2\), as ultrafilters are maximal and cannot properly contain one another.
           \end{itemize}
       \end{itemize}
       
       Therefore, in both cases, we have shown that if \(c_1\neq c_2\), then \(\Theta(c_1)\neq\Theta(c_2)\). Thus, \(\Theta\) is injective. This completes the proof for the injectivity of \(\Theta\).
   \end{subproof}

   Next, we prove that \(\Theta\) is surjective by demonstrating that for every ultrafilter \(f\in F\), there exists a context \(c_f\in \mathcal{C}\) such that \(\Theta(c_f)=f\).
   
   \begin{cor}
       \(\Theta\) is surjective.
   \end{cor}
   
   \begin{subproof}
       Surjectivity follows from Stone's Representation Theorem, which guarantees that every ultrafilter in \(F\) corresponds to some context \(c\in\mathcal{C}\). Specifically, for any ultrafilter \(f\in F\), there exists a unique context \(c_f\) such that \(F_{c_f}=f\). This is a direct consequence of the duality between Boolean algebras and Stone spaces. To demonstrate this, consider the following chain of arguments:
       \\\\First, we construct a context \(c_f\) corresponding to an arbitrary ultrafilter \(f\in F\), such that \(\Theta(c_f)=f\).

       \begin{enumerate}
           \item \textbf{Canonical Basis \(B_f\)}:
           First, identify a unique basis \(B_f\) for the context \(c_f\) such that \(\Theta(c_f)=f\). Consider the formulae within \(f\) that do not contain modal operators. These non-modal formulae correspond directly to propositions in a maximal Boolean sublattice of the orthomodular lattice \(\mathcal{L}\).

           We argue that all non-modal formulae within \(f\) must come from a single maximal Boolean sublattice \(B_f\). Suppose, for contradiction, that \(f\) contains non-modal formulae corresponding to propositions from two distinct maximal Boolean sublattices \(B_1\) and \(B_2\). This would imply the existence of basic open sets \(U_{c_1}\) and \(U_{c_2}\), corresponding to contexts \(c_1=(B_1,P_1)\) and \(c_2=(B_2,P_2)\), such that both \(U_{c_1}\in f\) and \(U_{c_2}\in f\).

           Consider two scenarios:
           \begin{itemize}
               \item If \(B_1\) and \(B_2\) correspond to mutually exclusive sublattices of \(\mathcal{L}\), their corresponding basic open sets \(U_{c_1}\) and \(U_{c_2}\) are disjoint. Therefore, \(U_{c_1}\cap U_{c_2}=\emptyset\), contradicting the fact that \(f\) is an ultrafilter, as ultrafilters cannot contain the empty set.
               \item If \(B_1\) and \(B_2\) have overlapping elements, then \(U_{c_1}\) and \(U_{c_2}\) might intersect non-trivially. However, this intersection would correspond to a context belonging to a common sublattice \(B_f\) that encapsulates this intersection. This still leads to a contradiction, as the non-modal formulae would need to be consistent within a single maximal Boolean sublattice \(B_f\).
           \end{itemize}

           Therefore, all non-modal formulae within \(f\) must correspond to propositions within a single maximal Boolean sublattice, which we identify as \(B_f\).

           \item \textbf{Construction of \(P_f\)}:
           Next, construct the set of propositions \(P_f\) for the context \(c_f=(B_f,P_f)\), where \(B_f\) is the previously identified canonical basis for \(f\).

           \begin{enumerate}
               \item \textbf{Direct Inclusion of Propositions}:
               Initialize \(P_f\) by including all propositions corresponding to non-modal clopen sets contained in \(f\):
               \[
               P_f=\{p\in B_f\mid\omega_p\in f,\text{ where }\omega_p\text{ is the clopen set corresponding to }p\}.
               \]
               \item \textbf{Inclusion from Modal Formulas}:
               Expand \(P_f\) by including propositions implied by modal formulas present in \(f\). For each modal formula \(\square^{B_i}_p\omega\in f\), include in \(P_f\) the propositions mapped from other bases \(B_i\) to \(B_f\) via the logical implications of these modal operators. Specifically, apply a unitary automorphism \(\Phi_{B_i\to B_f}\) to translate these propositions into the context of \(B_f\):
               \[
               P_f=P_f\cup\{\Phi_{B_i\to B_f}(p)\mid\square^{B_i}_p\omega\in f,\omega\text{ corresponds to }p\in B_i\}.
               \]
           \end{enumerate}
       \end{enumerate}

       Finally, verify that \(\Theta(c_f)=f\). We need to show that for every clopen set \(\omega\in V\), \(\omega\in\Theta(c_f)\) if and only if \(\omega\in f\). Since both \(F_{c_f}\) and \(f\) are ultrafilters, it suffices to prove one direction: \(F_{c_f}\subseteq f\). By showing this inclusion, we conclude that \(F_{c_f}=f\) because no ultrafilter can properly contain another.

       \begin{itemize}
           \item \textbf{Base Case (Non-Modal Formulae)}:
           For any non-modal formula \(\omega\in V\), if \(\omega\in F_{c_f}\), then by the construction of \(P_f\), the corresponding proposition \(p\in B_f\) is in \(P_f\). Since \(P_f\) was constructed to include all propositions corresponding to non-modal clopen sets in \(f\), it follows that \(\omega\in f\).

           \item \textbf{Inductive Step: Modal and Compound Formulae}
           \begin{itemize}
               \item \textbf{Negation}:
               Suppose \(\omega\) is a quantum formula such that \(\omega\notin F_{c_f}\). By the ultrafilter property, \(\neg\omega\in F_{c_f}\). Since \(\omega\notin F_{c_f}\), \(\omega\notin f\) by construction, hence \(\neg\omega\in f\) by the ultrafilter property.

               \item \textbf{Conjunction and Disjunction}:
               If \(\phi\) and \(\psi\) are quantum formulae such that \(\phi, \psi\in F_{c_f}\), then \(\phi\land\psi\in F_{c_f}\) by closure under intersections of \(F_{c_f}\). Since \(\phi,\psi\in F_{c_f}\), we know from the base case that \(\phi,\psi\in f\). Thus, \(\phi\land\psi\in f\) by closure under intersections of \(f\).

               \item \textbf{Modal Operators}:
               Suppose \(\omega=\square^{B_i}_p\omega_0\in F_{c_f}\), for some basis \(B_i\), proposition \(p\in c_f\), and quantum formula \(\omega_0\). From the definition of the constructed ultrafilter \(F_{c_f}\), the inclusion of \(\omega\) in \(F_{c_f}\) implies that for all contexts \(c'\) accessible from \(c_f\) via the relation \(R^{B_i}_p\), \(c'\in \omega_0\). This implies, by the nature of the modal operator and the definition of accessibility, that \(\Phi_{B_i\to B_f}(\omega_0)\land p\in c_f\), where \(\Phi_{B_i\to B_f}\) maps the contexts appropriately under the unitary transformation between bases. Define \(p'\equiv\Phi_{B_i\to B_f}(\omega_0)\land p\). Since \(p'\land p\) represents a logical consequence within the common basis \(B_f\) and is derived from the application of modal operators mapped correctly into the basis \(B_f\), the associated clopen set \(\omega_{p'\land p}\) must also be in \(f\) due to the base case. Hence, since \(f\) is closed under supersets and \(p' \land p \subseteq \square^{B_i}_p\omega_0\), it necessarily follows that \(\square^{B_i}_p\omega_0=\omega\in f\).
           \end{itemize}
       \end{itemize}

       By induction, we have shown that for any quantum formula \(\omega\in V\), if \(\omega\in F_{c_f}\), then \(\omega\in f\). Therefore, \(F_{c_f}\subseteq f\). Since both \(F_{c_f}\) and \(f\) are ultrafilters on \(V\), this implies that \(F_{c_f}=f\) since no ultrafilter can properly contain another. Thus, \(\Theta(c_f)=f\).

       This completes the proof of surjectivity, demonstrating that for every ultrafilter \(f\in F\), there exists a context \(c_f\in mathcal{C}\) such that \(\Theta(c_f)=f\).
   \end{subproof}

    \begin{lemma}
        \(\Theta\) is bijective.
    \end{lemma}
    \begin{subproof}
        By the previous two corollaries, we have that \(\Theta\) is both injective and surjective. Hence, it is bijective.
    \end{subproof}
   \subsection{Preservation of Logical Relations}
We have now validated that the mapping of contexts to ultrafilters is well-defined, injective, and surjective. To establish \(\Theta\) as a complete isomorphism, we proceed by verifying the preservation of the two logical relations, accessibility relations and valuations. We begin by demonstrating that \(\Theta\) preserves the accessibility relations defined on \(\mathcal{F}(\mathcal{L})\).

\begin{cor}
    \(\Theta\) preserves accessibility relations.
\end{cor}

\begin{subproof}
    This property can be formalized as \((c, c') \in R^{B_i}_p \Leftrightarrow (\Theta(c), \Theta(c')) \in {{R^{B_i}_p}^{+}}_+\), where \(R^{B_i}_p\) and \({{R^{B_i}_p}^{+}}_+\) are the accessibility relations on \(\mathcal{F}(\mathcal{L})\) and its double-dual, respectively. Recall that the accessibility relations on ultrafilters in the double-dual are defined as \((\omega,\omega')\in{{R^{B_i}_p}^{+}}_+\) iff \(\forall U\in V(\square^{B_i}_p U\in f\Rightarrow U\in f')\).

    \begin{itemize}
        \item \(\mathbf{\Rightarrow}\): We proceed with a proof by contradiction. Assume there exists a pair of contexts \(c,c'\in\mathcal{C}\) and a relation \(R^{B_i}_p\) such that \((c,c')\in R^{B_i}_p\) in the quantum frame, but \((\Theta(c), \Theta(c')) \notin {{R^{B_i}_p}^{+}}_+\) in the dual frame. This would imply that there exists some formula \(\omega'\in V\) such that \(\square^{B_i}_p\omega'\in\Theta(c)\) but \(\omega'\notin\Theta(c')\). In the quantum frame, this is equivalent to \(U_c\in\square^{B_i}_p\omega'\) but \(U_{c'}\notin\omega'\). However, by definition of the accessibility relation on \(\mathcal{F}(\mathcal{L})\),\(U_{c'}\in\omega'\), which leads to a contradiction. Thus, our initial assumption was false, and we must have \((c,c')\in R^{B_i}_p\Rightarrow(\Theta(c),\Theta(c'))\in {{R^{B_i}_p}^{+}}_+\).

        \item \(\mathbf{\Leftarrow}\): Again, we use proof by contradiction. Assume that \((\Theta(c),\Theta(c'))\in{{R^{B_i}_p}^{+}}_+\) but \((c,c')\notin R^{B_i}_p\). The latter condition implies that \(c'\notin\mathcal{A}^i_1\). However, the first condition implies that \(c'\) is in the basis \(B_i\), leading to two possibilities:
        \begin{enumerate}
            \item \(\mathbf{B=B_i}\): Since \((\Theta(c),\Theta(c'))\in R^{B_i}_p\) and \(c\) is assumed to be in the same basis \(B_i\), defined by the modal operator \(\square^{B_i}_p\) associated with \(R^{B_i}_p\), we have \(U_c\in\omega\Leftrightarrow U_c\in\square^{B_i}_p\omega\) for all non-modal formulae \(\omega\). Substituting this equivalence into the definition of accessibility in the double-dual, we get that \((\Theta(c),\Theta(c'))\in R^{B_i}_p\Leftrightarrow\forall\omega\in V(\omega\in\Theta(c)\Rightarrow\omega\in\Theta(c'))\). Consequently, \(\Theta(c)\subseteq\Theta(c')\). However, since \(\Theta(c)\) and \(\Theta(c')\) are ultrafilters on \(V\), they are upward-closed. Given that an ultrafilter cannot be properly contained by any other filter, the only possibility is that \(\Theta(c)=\Theta(c')\). By the injectivity of \(\Theta\), these contexts must be equivalent. Since \(\mathcal{A}_n^B\) iteratively includes all previous sets of accessible contexts, we conclude that any context is accessible to itself, thus \((c,c')\in R^{B_i}_p\), and our initial assumption was false.

            \item \(\mathbf{B\neq B_i}\): Assuming \((c,c')\notin R^{B_i}_p\), we have \(c'\notin F_{c''}^{max}\), meaning there exists some proposition in \(c'\) that is not compatible with \(c''\), where \(c''=(B,P)\). If \(c\in\square^{B_i}_p\omega\), then all contexts in \(F_{c''}^{max}\) in the basis \(B_i\) satisfy \(\omega\). In other words, for all \(c_1,c_2,c_3,\cdots\in F_{c''}^{max}\), \(\omega\in (Th(c_1)\cap Th(c_2)\cap Th(c_3)\cap\cdots)\). Since \((\Theta(c),\Theta(c'))\in{{R^{B_i}_p}^{+}}_+\) implies \(\square^{B_i}_p\omega\Rightarrow\omega\in\Theta(c')\), it must be the case that \((Th(c_1)\cap Th(c_2)\cap Th(c_3)\cap\cdots)\subseteq Th(c')\), which in turn implies that \(c'\) is indeed in \(F_{c''}^{max}\), and thus \((c,c')\in R^{B_i}_p\).
        \end{enumerate}
    \end{itemize}
    
    This concludes the proof of the preservation of accessibility relations, establishing that \((c,c')\in R^{B_i}_p\) if and only if \((\Theta(c),\Theta(c'))\in{{R^{B_i}_p}^{+}}_+\).
\end{subproof}

    Finally, we demonstrate the preservation of valuations. 
\begin{cor}
    \(\Theta\) preserves valuations.
\end{cor}

\begin{subproof}
    To prove that \(\Theta\) preserves valuations, we must show that for any clopen subset \(S\subseteq\mathcal{C}(\mathcal{L})\), the image \(\Theta(S)\) is also a clopen subset in \(F\), the double-dual frame. This will be accomplished by demonstrating that \(\Theta\) preserves both openness and closedness.

    \begin{itemize}
        \item \textbf{Preserves Openness}: Let \(S\subseteq\mathcal{C}(\mathcal{L})\) be a clopen subset. We need to show that \(\Theta(S)\) is open in \(F\). This involves demonstrating that for each ultrafilter \(f\in\Theta(S)\), there exists a basic open set \(V_f\subseteq F\) such that \(f\in V_f\subseteq\Theta(S)\).
        \begin{enumerate}
            \item \textbf{Existence of Basic Open Sets in \(\mathcal{C}(\mathcal{L})\)}: Since \(S\) is open in \(\mathcal{C}(\mathcal{L})\), for each context \(c\in S\), there exists a basic open set \(U_c\subseteq\mathcal{C}(\mathcal{L})\) such that \(c\in U_c\subseteq S\). The set \(U_c\) is open by definition and includes all contexts sharing a specific logical property or set of propositions with \(c\).
            \item \textbf{Mapping Under \(\Theta\)}: By the definition of \(\Theta\), the ultrafilter \(f=\Theta(c)\) corresponds to the filter generated by \(U_c\). This implies that \(f\) contains all the clopen sets in \(V\) (the set of clopen subsets of \(F\)) that correspond to the propositions satisfied by \(c\) in \(U_c\).
            \item \textbf{Definition of \(V_f\)}: Define \(V_f\) as the basic open set in \(F\) corresponding to the clopen set \(U_c\). Specifically, \(V_f\) consists of all ultrafilters in \(F\) that contain \(U_c\). Formally,
        \[
        V_f=\{f'\in F\mid U_c\in f'\}.
        \]
        Since \(U_c\subseteq S\), it follows that every ultrafilter \(f'\) in \(V_f\) will contain the set \(U_c\), and hence \(f' \in \Theta(S)\). Therefore, \(V_f\subseteq\Theta(S)\).
        \end{enumerate}
        Now, we verify openness. Since \(f=\Theta(c)\) contains \(U_c\) by construction, \(f\in V_f\). As \(V_f\) is a basic open set in \(F\) and \(V_f\subseteq\Theta(S)\), this demonstrates that \(\Theta(S)\) is open in \(F\). Thus, \(\Theta\) preserves the openness of clopen sets.

        \item \textbf{Preserves Closedness}

        To show that \(\Theta\) preserves closedness, we need to demonstrate that if \(S\subseteq\mathcal{C}(\mathcal{L})\) is closed, then \(\Theta(S)\) is also closed in \(F\).
        \begin{enumerate}
            \item \textbf{Complement of \(S\)}: Since \(S\) is closed in \(\mathcal{C}(\mathcal{L})\), its complement \(S^c=\mathcal{C}(\mathcal{L})\setminus S\) is open in \(\mathcal{C}(\mathcal{L})\). As shown in the previous step, \(\Theta(S^c)\) is open in \(F\).
            \item \textbf{Complement of \(\Theta(S)\)}: The complement of \(\Theta(S)\) in \(F\) is given by \(\Theta(S)^c=F\setminus \Theta(S)\). Since \(\Theta\) is a bijection, \(\Theta(S^c)=\Theta(\mathcal{C}(\mathcal{L})\setminus S)\), which represents the set of all ultrafilters corresponding to contexts in \(S^c\).
            \item \textbf{Openness of \(\Theta(S^c)\)}: As shown earlier, \(\Theta(S^c)\) is open in \(F\), which implies that the complement of \(\Theta(S)\) is open. Therefore, \(\Theta(S)\) is closed in \(F\).
        \end{enumerate}
        \end{itemize}
        Since \(\Theta(S)\) has been shown to be both open and closed in \(F\), it follows that \(\Theta(S)\) is a clopen subset of \(F\). Therefore, \(\Theta\) preserves clopen subsets, ensuring that it preserves valuations. This completes the proof that \(\Theta\) preserves valuations.
\end{subproof}

    With the isomorphism between the quantum frame \(\mathcal{F}(\mathcal{L})\) and its double dual \(\mathcal{F}(\mathcal{L})^{++}\) established, we can conclude that the internal logic of the QCT, interpreted in the context of the quantum frame, corresponds to propositional polymodal logic. The modal operators in the logic reflect the accessibility relations \(R^j_i\) in the quantum frame, and the preservation of these relations by the isomorphism \(\Theta\) ensures that the logical structure is maintained.
    \end{proof}
\textbf{}
\\\\We have now demonstrated that the internal logic of the Quantum Contextual Topos is fully aligned with propositional polymodal logic. Through the construction of the double-dual frame and the verification of the isomorphism \(\Theta\), we have shown that the logical structures and accessibility relations within any quantum frame \(F(L)\) are preserved under dualization. This preservation affirms that the QCT framework faithfully encapsulates the logical intricacies of quantum propositions within a topological and modal structure. The results solidify the QCT as a robust and versatile mathematical tool, capable of bridging the gap between classical and quantum logic, while providing a coherent environment for the exploration of contextual quantum reasoning. The rigorous alignment with polymodal logic not only validates the internal consistency of the QCT but also enhances its potential applicability in quantum foundations and related fields.
\section{Conclusion}

In this paper, we have introduced the Quantum Contextual Topos (QCT), a novel framework designed to address the challenges posed by traditional quantum logic. By embedding quantum propositions within a topos-theoretic structure, QCT provides a rigorous mathematical foundation capable of preserving classical logic rules when applied to specific quantum contexts. This framework bridges the gap between classical logic and quantum mechanics, allowing for a more consistent interpretation of quantum phenomena.

Our work focused on the construction of the QCT, emphasizing its internal logic and the role of contextuality within quantum systems. We demonstrated that the QCT framework not only aligns with the principles of quantum mechanics but also retains the logical consistency required for classical reasoning. The key contributions of this paper include the definition of the Quantum Frame, the development of contextual topologies, and the formalization of quantum formulae within this structure.

Future work will involve exploring the practical implications of the QCT framework, particularly in the context of quantum computing.

\bibliography{bibl}
\end{document}